\documentclass{article}

\usepackage[utf8]{inputenc}
\usepackage[a4paper,bottom=4cm]{geometry}
\usepackage{amsmath, amssymb}
\usepackage{enumitem}
\usepackage{microtype}
\usepackage{bbm}
\usepackage{multirow}

\usepackage[backend=biber, style=alphabetic, maxbibnames=99, giveninits=true]{biblatex}
\renewbibmacro{in:}{}
\AtEveryBibitem{\clearfield{month}} 
\AtEveryBibitem{\clearfield{day}}
\AtEveryBibitem{\clearfield{isbn}}
\AtEveryBibitem{\clearfield{issn}}
\AtEveryBibitem{\clearlist{language}}
\renewbibmacro*{doi+eprint+url}{%
    \printfield{doi}%
    \newunit\newblock%
    \iftoggle{bbx:eprint}{%
        \usebibmacro{eprint}%
    }{}%
    \newunit\newblock%
    \iffieldundef{doi}{%
        \usebibmacro{url+urldate}}%
        {}%
    }
\usepackage[usenames,dvipsnames]{xcolor}
\definecolor{myred}{RGB}{255, 61, 65}

\usepackage{mathtools}
\usepackage[usenames,dvipsnames]{xcolor}
\definecolor{myred}{RGB}{255, 61, 65}
\usepackage{amsthm}
\usepackage{asymptote}
\usepackage[plainpages=false,pdfpagelabels,colorlinks=true,citecolor=blue,hypertexnames=false]{hyperref}
\usepackage[labelfont={sc}]{caption}

\usepackage[capitalize]{cleveref}
\usepackage{todonotes}

\usepackage[all]{xy}
\usepackage{tikz}
\usetikzlibrary{cd}

\usepackage[capitalize]{cleveref}
\usepackage{colonequals}

\usepackage{multicol}
\usepackage{pifont}
\usepackage{makecell}

\usetikzlibrary{shapes.geometric, arrows, decorations.pathmorphing, arrows.meta}

\tikzstyle{startstop} = [rectangle, rounded corners, minimum width=3cm, minimum height=1cm,text centered, draw=black, fill=brown!30]
\tikzstyle{decision} = [diamond, aspect=1.5, inner xsep=0pt, text centered, draw=black, fill=lime!30, text width=1.8cm]
\tikzstyle{processyes} = [rectangle, minimum width=3cm, minimum height=1cm, text centered, draw=black, fill=green!30]
\tikzstyle{processno} = [rectangle, minimum width=3cm, minimum height=1cm, text centered, draw=black, fill=red!30]
\tikzstyle{arrow} = [thick,-latex,>=stealth]

\usepackage{array}

\usepackage{todonotes}

\renewcommand{\phi}{\varphi}
\renewcommand{\epsilon}{\varepsilon}
\renewcommand{\theta}{\vartheta}

\newcommand{\Z}{\mathbb{Z}}

\newcommand{\F}{\mathbb{F}}

\newcommand{\Fp}{\mathbb{F}_p}

\newcommand{\ps}[1]{[\![#1]\!]}

\DeclareMathOperator{\Gal}{Gal}

\def\pFq#1#2#3#4#5{{}_{#1}F_{#2}\left(#3, #4; #5\right)}

\newcommand{\legendre}[2]{\left(\!\frac{#1}{#2}\!\right)}

\theoremstyle{definition}
\newtheorem{defi}{Definition}

\theoremstyle{plain}

\newtheorem{thm}[defi]{Theorem}
\newtheorem{lem}[defi]{Lemma}

\newtheorem{prop}[defi]{Proposition}

\theoremstyle{remark}

\title{Galois Groups of Apéry-like Series Modulo Primes}
\author{Xavier Caruso, Florian Fürnsinn, Daniel Vargas-Montoya and Wadim Zudilin}
\date\today

\begin{document}
\maketitle

\begin{abstract}
We compute the Galois groups of the reductions
modulo the prime numbers $p$ of the generating series of 
Apéry numbers, Domb numbers and Almkvist--Zudilin numbers.
We observe in particular that their behavior is governed
by congruence conditions on $p$.
\end{abstract}

\section{Introduction}

The \emph{Apéry numbers} \begin{equation*}
\alpha_n = \sum_{k=0}^n{\binom nk}^2{\binom{n+k}n}^2 \quad\text{for}\; n=0,1,2,\dots
\end{equation*} are a famous sequence of integer numbers, mostly known for playing a prominent role in Apéry's proof \cite{Ape79} of the irrationality of $\zeta(3)$. Their generating series $f_\alpha\coloneqq \sum_{n=0}^\infty \alpha_n t^n$  enjoys many interesting properties. It is \emph{D-finite}, i.e., it satisfies a linear differential equation with polynomial coefficients. Starting from the differential equation, integrality of the coefficients of its solutions is highly remarkable. Further, the Apéry numbers grow in a controlled manner, making their generating series a \emph{G-function}. Moreover, letting $p$ be an odd prime number, it is known \cite[Theorem 1]{Gessel} that the Apéry numbers have the 
\emph{$p$-Lucas property} for all prime numbers $p$ (see also the general 
result for Apéry-like numbers in \cite{Malik-Straub}). That is, we have 
$\alpha_{np+\ell}\equiv \alpha_n \alpha_\ell \pmod p$ whenever $0 \leq 
\ell < p$. On the level of generating functions this property translates 
to the congruence $f_\alpha \equiv A_p\cdot f_\alpha^p \pmod p$, where 
$A_p \coloneqq \sum_{n=0}^{p-1} \alpha_n t^n$ denotes the truncation 
of $f_\alpha$ at order $p$. This relation also shows that the 
reduction of the generating function $f_\alpha\bmod p$ is algebraic 
over the field of rational functions $\F_p(t)$ and that the extension
it generates is Kummer. 

Algebraicity modulo (almost) all primes $p$ is a phenomenon observed for many classes of D-finite series. More precisely, for \emph{diagonals} of multivariate rational, or, equivalently, algebraic functions this is a consequence of a theorem of Furstenberg~\cite{Fur67}. For \emph{hypergeometric functions} this can be deduced from work by Christol~\cite{Chr86a} and was made explicit by Vargas-Montoya~\cite{Var21}. Further, if Christol's Conjecture~\cite{Chr86} proves to be true, the result of Furstenberg would apply to all \emph{globally bounded} D-finite series.

Having algebraic equations for the generating function modulo prime numbers, it is natural to consider their Galois groups for each prime number $p$. For D-finite series whose reductions modulo infinitely many prime numbers $p$ are algebraic, in   \cite{CFV25} it was (conjecturally) observed that these Galois groups show some uniformity properties across different primes, and seem to be related to the \emph{differential Galois group} of the minimal differential equation satisfied by the series. As an example in \cite[\S~2.1.5]{CFV25}, the Galois groups of the reductions of the Apéry series were computed for many prime numbers using a computer algebra system. However, the pattern that unfolded was stated without a proof, and the aim of this text is to provide one.

Fixing a prime number $p$ and starting from the algebraic relation given by the $p$-Lucas property, the Galois group of $f_\alpha \bmod p$ can canonically be 
seen as a subgroup of $\F_p^\times$ \emph{via} the embedding
$$\begin{array}{rcl}
\Gal\big(\F_p(t,f_\alpha)/\F_p(t)\big) & \longrightarrow & \F_p^\times \\
\sigma & \mapsto & \sigma(f_\alpha)/f_\alpha.
\end{array}$$
We aim to determine the image of the above map.
Throughout the article, we denote by $S$ the unique subgroup of $\F_p^\times$ of index $2$; it is its subgroup of squares, hence the notation.

\begin{thm}~
\label{thm:apery}
\begin{itemize}
\item If $p \equiv 1, 5, 7, 11 \pmod{24}$, then
$\Gal\big(\F_p(t, f_\alpha)/\F_p(t)\big) = S$.

\item If $p \equiv 13, 17, 19, 23 \pmod{24}$, then
$\Gal\big(\F_p(t, f_\alpha)/\F_p(t)\big) = \F_p^\times$.
\end{itemize}
\end{thm}

\noindent
This result is the Galois theoretic shadow of the following
factorization property of $A_p$.

\begin{thm}
\label{thm:Ap}
There exists a polynomial $B_p \in \Fp[t]$ such that
\begin{itemize}
\item if $p \equiv 1, 5, 7, 11 \pmod{24}$, then $A_p = B_p^2$,
\item if $p \equiv 13, 17, 19, 23 \pmod{24}$, then $A_p = (t^2 - 34t + 1) \cdot B_p^2$.
\end{itemize}
\end{thm}

The main ingredient in our proof is the fact that the Apéry series $f_\alpha$ is related to the generating function of the \emph{Franel numbers} \[h\coloneqq \sum_{n=0}^\infty \sum_{k=0}^{n}\binom{n}{k}^3x^n\] by a simple rational substitution, namely
$f_\alpha = (1+x)\cdot h^2$ where the variables $t$ and $x$ are linked by the relation
$t = \frac {x(1-8x)}{1+x}$.

The sequence of Ap\'ery numbers has two companions:
the (alternating version of the) Domb numbers
\begin{equation*}
\delta_n = (-1)^n\sum_{k=0}^n \binom{2k}k\binom{2n-2k}{n-k}{\binom nk}^2
\quad\text{for}\; n=0,1,2,\dots
\end{equation*}
and the Almkvist--Zudilin numbers
\begin{equation*}
\xi_n=\sum_{k=0}^n(-1)^{n-k}3^{n-3k}\frac{(3k)!}{k!^3}\binom n{3k}\binom{n+k}n
\quad\text{for}\; n=0,1,2,\dots,
\end{equation*}
in the following sense:
These three sequences satisfy similar difference equations of order~2 and degree~3, and their generating series
\begin{equation*}
f_\alpha = \sum_{n=0}^\infty\alpha_n t^n,
\quad
f_\delta = \sum_{n=0}^\infty\delta_n t^n
\quad\text{and}\quad
f_\xi = \sum_{n=0}^\infty\xi_n t^n
\end{equation*}
admit modular parameterizations via the Hauptmoduln of the three subgroups of index~2
lying between $\Gamma_0(6)$ and its normalizer in~$\operatorname{SL}_2(\mathbb R)$
(see \cite{Chan-Verrill} and \cite{Chan-Zudilin}).

Thanks to \cite[Theorem 2.2]{Chan-Zudilin}, the argument for $f_\alpha$ extends to the series $f_\delta$ and $f_\xi$ with the new relations:
\begin{align*}
f_\delta = (1-8x) \cdot h^2 & \quad \text{with} \quad t = \frac{x(1+x)}{1-8x}, \\
f_\xi = (1+x)(1-8x) \cdot h^2 & \quad \text{with} \quad t = \frac{x}{(1+x)(1-8x)} 
\end{align*}
respectively.

\begin{thm}
\label{thm:domb}
Set
\[
A_{\delta,p}=\sum_{n=0}^{p-1}\delta_nt^n\in\mathbb F_p[t].
\]
Then there exists a polynomial $B_{\delta,p} \in \F_p[t]$ such that:
\begin{itemize}
\item if $p \equiv 1 \pmod{6}$, then $A_{\delta,p} = B_{\delta,p}^2$ and $\Gal\big(\F_p(t, f_\delta)/\F_p(t)\big) = S$,
\item if $p \equiv 5 \pmod{6}$, then $A_{\delta,p} = (64t^2-20t+1) B_{\delta,p}^2$ and $\Gal\big(\F_p(t, f_\delta)/\F_p(t)\big) = \F_p^\times$.
\end{itemize}
\end{thm}

\begin{thm}
\label{thm:az}
Set
\[
A_{\xi,p}=\sum_{n=0}^{p-1}\xi_nt^n\in\mathbb F_p[t].
\]
Then there exists a polynomial $B_{\xi,p} \in \F_p[t]$ such that:
\begin{itemize}
\item if $p \equiv 1,3 \pmod{8}$, then $A_{\xi,p} = B_{\xi,p}^2$ and $\Gal\big(\F_p(t, f_\xi)/\F_p(t)\big) = S$,
\item if $p \equiv 5,7 \pmod{8}$, then $A_{\xi,p} = (81t^2+14t+1) B_{\xi,p}^2$ and $\Gal\big(\F_p(t, f_\xi)/\F_p(t)\big) = \F_p^\times$.
\end{itemize}
\end{thm}

In Section~\ref{ssec:zoo} we collect more examples of series, where we computationally observed similar patterns.

\paragraph{Acknowledgments}
We thank Alin Bostan for pointing out the observation that the Domb numbers and the AZ-numbers share the same behaviour as the Apéry numbers.

F.F.\ was funded by a DOC Fellowship (27150) of the \href{https://www.oeaw.ac.at/en/}{Austrian Academy of Sciences} at the University of Vienna. Further he thanks the French–Austrian project EAGLES (ANR-22-CE91-0007 \& FWF grant \href{https://doi.org/10.55776/I6130}{10.55776/I6130}) for financial support.

X.C., F.F.\ and D.V.-M.\ thank \href{https://oead.at/en/}{Austria’s Agency for Education and Internationalisation (OeAD)} and Campus France for providing funding for research stays via WTZ collaboration project/Amadeus project FR02/2024.
F.F.\ and W.Z.\ thank the Max Planck Institute for Mathematics (Bonn, Germany) for their hospitality and financial support in May 2025, where the initial discussion on the project commenced.
The work of W.Z.\ was supported in part by the NWO grant OCENW.M.24.112.

\section{The Apéry Numbers}

Throughout this section, we fix an odd prime number $p$ and we
set $t\coloneqq \frac {x(1-8x)}{1+x}$.

\begin{lem} \label{lem:extensionxt}
The extension $\F_p(x)/\F_p(t)$ is a quadratic 
extension, generated by $(t^2-34t+1)^{1/2},$ with nontrivial element $\sigma:x\mapsto \frac{1-8x}{8+8x}$ of the Galois group.
\end{lem}

\begin{proof}
     We note that $x$ is a solution of the quadratic equation $8x^2+(t-1)x+t=0$ over $\F_p(t)$ with discriminant $\Delta=t^2-34t+1$. An easy calculation shows that $\frac{1-8x}{8+8x}$ also solves this equation.
\end{proof}

Let \[h\coloneqq \sum_{n=0}^\infty \sum_{k=0}^{n}\binom{n}{k}^3x^n \in \Fp\ps{x} \qquad \text{and} \qquad H\coloneqq \sum_{n=0}^{p-1} \sum_{k=0}^{n}\binom{n}{k}^3 x^n \in \Fp[x].\] 
By Lucas' Theorem (on evaluating binomial coefficients modulo $p$) one easily checks that the coefficients of $h$ satisfy Lucas' congruences and we deduce that 
$h=H h^p$. Thus $h^2 = H^{-1/e}$ with $e \coloneq (p{-}1)/2$.

\begin{lem}
 The polynomial $H$ is not an $n$-th power in $\F_p[x]$ for $n\geq 2$.
\end{lem}
\begin{proof}
The series $h$ satisfies the second order differential equation 
\begin{equation} \label{eq:deq:h} \textstyle
    x(x+1)(8x-1)\frac{d^2h}{dx^2} + (24x^2+14x-1) \frac{dh}{dx} + (8x+2)h = 0. 
\end{equation}
By \cite[Lemma~2.1.3]{CFV25}, $H$ is not an $e$-th power of a polynomial for any $e$, as soon as $H$ has a root different from $0, -1, 1/8$. One checks that $H$ has degree $p{-}1$ and using that $\binom{p{-}1}{n}\equiv (-1)^n \pmod p$ that it has leading coefficient $1$. Also, it is clear that its constant coefficient is $1$, so it does not have $0$ as a root. Assuming that 
\[\textstyle H=\big(x-\frac 1 8\big)^m (x+1)^{p-1-m}\] for some $m$, we deduce by comparing the constant coefficients that $(-1/8)^m\equiv 1 \pmod p$. Comparing the coefficient of $x$ one obtains that
\[m\left(-\frac 1 8 \right) ^{m-1}+ (p-1-m) \equiv 2 \pmod p,\]
so we deduce $m=(p{-}1)/3$ if $p\equiv 1 \pmod 3$ and $m=(2p{-}1)/3$ otherwise.
If $p\equiv 1 \pmod 3$, we can then compare coefficients of $x^2$ and find the contradiction $13\equiv10\pmod p$.
If $p\equiv 2 \pmod 3$, the contradiction already arises at the constant coefficient.
Thus $H$ has a root different from $-1$ and $1/8$, and the claim follows.
\end{proof}

Since $\F_p(x)$ contains all $e$-roots of unity (as they are already in $\F_p$), Kummer theory implies that the extension $\F_p(x, h^2)/\F_p(x)$ is an abelian extension of degree $e$ whose Galois group is canonically isomorphic to $S$.

We now consider the tower of extensions:

\begin{center}
\begin{tikzpicture}[yscale=1.5]
\node(Fpt) at (0,0) { $\F_p(t)$ };
\node(Fpx) at (0,1) { $\F_p(x)$ };
\node(Ext) at (0,2) { $\F_p(x,h^2)$ };
\node[right] at (0.6,2) { $ = \F_p\big(x,\sqrt[e] H\big)$ };
\draw(Fpt)--(Fpx) node[right,midway,scale=0.9] { $\Gal = \left<\sigma\right>$ };
\draw(Fpx)--(Ext) node[right,midway,scale=0.9] { $\Gal = S$ };
\end{tikzpicture}
\end{center}

\noindent
and aim at studying the Galois properties of the total extension $\F_p(x,h^2)/\F_p(t)$.
In order to do this, we first determine the action of $\sigma$ on $H$.

\begin{lem} \label{lem:sigmaH}
    We have 
    \[H=\begin{cases}
        \sigma(H)\cdot (x+1)^{p-1} & \text{if $p\equiv 1 \pmod 6$}\\
      - \sigma(H)\cdot (x+1)^{p-1} & \text{if $p\equiv 5 \pmod 6$.}
    \end{cases}\]
\end{lem}
\begin{proof}
    We first notice that $\sigma$ is an involution and that $G=\sigma(H) \cdot (x+1)^{p-1}$ is a polynomial of degree $p{-}1$.
    Besides, one checks that $H$ is a solution of the differential equation~\eqref{eq:deq:h}.
    Expressing now $G$, $\frac{dG}{dx}$ and $\frac{d^2G}{dx^2}$ as $\F_p(x)$-linear combinations of $\sigma(H)$ and its successives derivatives (just by applying the chain and product rules), we find that $G$ is also a solution of the same differential equation.
    Therefore, we must have
    \begin{equation} \label{eq:sigmaH} 
      H=aG = a\cdot \sigma(H) \cdot (x+1)^{p-1}
    \end{equation}
    with $a\in \F_p$.
    Recalling that 
    \[\sigma(H)=\sum_{n=0}^{p-1}\sum_{k=0}^n \binom{n}{k}^3\left(\frac{1-8x}{8+8x}\right)^n\]
    we evaluate Equation~\eqref{eq:sigmaH} at $x=-1$ to obtain $H_{|x=-1}=a$.
    We set $y \coloneqq \frac{ 27 x^2}{(1 - 2x)^3}$ and introduce the
    hypergeometric series
    \[g \coloneqq \pFq 2 1 {[1/3, 2/3]} {[1]} {y} = 
      \sum_{k=0}^{+\infty} \frac{(1/3)_k (2/3)_k}{k!^2} y^k \in \F_p\ps y \subset \F_p\ps{x}\]
    where $(\alpha)_k \coloneq \alpha (\alpha{+}1) \cdots (\alpha{+}k{-}1)$ denotes the \emph{Pochhammer symbol}.
    We let also $G$ denote the truncation at $y^p$ of $g$. We know that
    $h=\frac{g}{1-2x}$, see for example \cite[Ex.~3.4]{SZ15}.  Moreover $g$ is $p$-Lucas as a series in $y$ and thus $G = g^{1-p}$. 
    From the fact that $h$ is also $p$-Lucas, we get
    \[H=h^{1-p} = \frac{(1-2x)^{p-1}}{g^{p-1}}=(1-2x)^{p-1}G. \]
    Evaluating this expression at $x=-1$ gives $a = G_{|x=-1} = G_{|y=1}$.   
    
    Thus, if $p\equiv 1 \pmod 3$, \emph{i.e.}, $p\equiv 1 \pmod 6$ we have
    \[a\equiv  \sum_{k=0}^{p-1}\binom{-1/3}{k}\binom{-2/3}{k} \equiv \sum_{k=0}^{p-1}\binom{(p{-}1)/3}{k}\binom{(2p-2)/3}{k}\equiv\binom{p{-}1}{(p{-}1)/3}\equiv 1 \pmod p.\]
    If $p\equiv 2 \pmod 3$, \emph{i.e.}, $p\equiv 5 \pmod 6$ we then obtain similarly
    \[a\equiv  \sum_{k=0}^{p-1}\binom{-1/3}{k}\binom{-2/3}{k} \equiv \sum_{k=0}^{p-1}\binom{(2p{-}1)/3}{k}\binom{(p-2)/3}{k}\equiv\binom{p{-}1}{(p{-}1)/3}\equiv -1 \pmod p.\]
    We have used the Chu-Vandermonde identity $\sum_{k=0}^r\binom{r}{k}\binom{s}{k}=\binom{r+s}{r}$ and the classical congruence $\binom{p{-}1}{k}\equiv (-1)^k \pmod p$.    
\end{proof}

\begin{prop}
    The field extension $\F_p(x,h^2)/\F_p(t)$ is an abelian Galois extension.
\end{prop}
\begin{proof}
    We first determine all the prolongations of $\sigma : \F_p(x) \to \F_p(x)$ to an automorphism of $\F_p(x, h^2)$.
    Such a prolongation is uniquely determined by the value of $\sigma(h^2)$, which needs to satisfy $\sigma(h^2)^e=\sigma(H)^{-1}$. Thus,
    \begin{equation} \label{eq:extensionsigma}
        \sigma(h^2) = u \cdot h^2 \cdot (x+1)^2,
    \end{equation}
    for some $u \in \F_p$.
    Besides, according to Lemma~\ref{lem:sigmaH}, $u$ has to be a square if $p\equiv 1 \pmod 6$ and not a square otherwise.
    Consequently, there are precisely $e$ prolongations of the automorphism~$\sigma$. There are clearly also $e$ prolongations of the identity (given by the group $S$ of squares in $\F_p$ acting by multiplication on $h^2$). Hence, the extension $\F_p(x,h^2)/\F_p(t)$ is Galois. Moreover, any prolongation of $\sigma$ clearly commutes with any element of $S$, making the Galois group abelian.
\end{proof}

\begin{lem} \label{lem:extsigma}
    There exists a prolongation of $\sigma$ to $\F_p(x, h^2)$ that is an involution, if and only if one can take $u=\pm\frac 8 9$ in Equation~\eqref{eq:extensionsigma}.
\end{lem}
\begin{proof}
    We have 
    \[\sigma^2(h^2)=  \sigma(u h^2 (x+1)^2) = u^2 h^2 (x+1)^2 (\sigma(x)+1)^2= u^2 \left(\frac{9}{8}\right)^2 h^2.\] The claim follows.
\end{proof}

\begin{prop} \label{prop:Galh}
    The Galois group of $\F_p(x,h^2)/\F_p(t)$ is 
    \[\Gal(\F_p(x,h^2)/\F_p(t)) = \begin{cases}
        \Z/(p{-}1)\Z\simeq \Z/e\Z\times \Z/2\Z & \text{if } p\equiv 3 \pmod 4\\
        \Z/(p{-}1)\Z & \text{if } p\equiv 13, 17 \pmod {24}\\
        \Z/e\Z\times \Z/2\Z& \text{if } p\equiv 1, 5 \pmod {24}.
    \end{cases}
    \]
    This corresponds to the cases where precisely one of the two numbers $u =\pm \frac 8 9$ corresponds to an involution $\sigma: h^2\mapsto u h^2 (1+x)^2$ in $\Gal(\F_p(x, h^2)/\F_p(t)$, none of them do, respectively both of them do.
\end{prop}

\begin{proof}
Note that $\frac 8 9 = 2 \cdot \left(\frac 2 3\right)^2.$ Moreover, $2$ is a square modulo $p$ if and only if $p\equiv \pm 1 \pmod 8$ and $-1$ is a square modulo $p$ if and only if $p\equiv 1 \pmod 4$. 

If $p\equiv 1 \pmod 6$ then, for any prolongation of $\sigma$, as in Equation~\eqref{eq:extensionsigma}, $u$ must be a square; if $p\equiv 5 \pmod 6$ the opposite must hold. As for $p\equiv 3 \pmod 4$, precisely one of the two values of $\pm \frac 8 9$ is a square, we always have precisely one adequate choice for $u$. If $p\equiv 1 \pmod 4$, either both of them are a square, or none of them are. 
\begin{itemize}
    \item If $p\equiv 1 \pmod{24}$, then $u$ has to be a square, and $\pm \frac 8 9$ both are squares.
    \item If $p\equiv 5 \pmod{24}$, then $u$ has to be a non-square, and $\pm \frac 8 9$ both are non-squares.
    \item If $p\equiv 13 \pmod{24}$, then $u$ has to be a square, but $\pm \frac 8 9$ both are non-squares.
    \item If $p\equiv 17 \pmod{24}$, then $u$ has to be a non-square, but $\pm \frac 8 9$ both are squares.
\end{itemize}
This concludes the proof of the second statement. 

For the first statement, we note that $\Gal(\F_p(x,h^2)/\F_p(t))$ always contains a cyclic group of order $e$, corresponding to $\Gal(\F_p(x,h^2)/\F_p(x)) \simeq S$.
This thus leaves us only with two possible choices, either $\Z/(p{-}1)\Z$ or $\Z/e\Z\times \Z/2\Z$ (which actually collapse in the case where $p\equiv 3 \pmod 4$). The direct factor $\Z/2\Z$ occurs if and only if there exists an element of the Galois group of order $2$, which does not belong to $S$, \emph{i.e.}, a prolongation of $\sigma$ which is an involution. We conclude using Lemma~\ref{lem:extsigma}.
\end{proof}

We recall that $f=h^2\cdot (x+1)$. This shows that $\F_p(t, f)\subseteq \F_p(x, h^2).$ Moreover $\F_p(t, f)$ is a cyclic extension of $\F_p(t)$, of degree $p{-}1$ or $e$, as explained in \cite[Subsection~2.1.5]{CFV25}.

\begin{proof}[Proof of Theorem~\ref{thm:apery}.]
    We distinguish cases according to the congruence class of $p$ modulo $24$:

    \begin{itemize}
        \item If $p\equiv 1 \pmod{24}$, the Galois group of $\F_p(x, h^2)/\F_p(t)$ is not cyclic, according to Proposition~\ref{prop:Galh}. Thus, $\F_p(t, f)$ is a proper subfield of $\F_p(x, h^2)$, necessarily of degree $e$ over $\F_p(t)$.
        \item If $p\equiv 5 \pmod{24}$ the Galois group of $\F_p(x, h^2)/\F_p(t)$ is not cyclic either, according to Proposition~\ref{prop:Galh}. We conclude again that $\F_p(f,t)$ has degree $e$ over $\F_p(t)$.
        \item If $p\equiv 7 \pmod{24}$, the Galois group of $\F_p(x, h^2)/\F_p(t)$ is cyclic, because $e$ is odd. Then, we need to check whether $f$ is in the unique subfield of $\F_p(x, h^2)$ of index 2. This is equivalent to checking whether it is invariant under the unique element  of the Galois group of order $2$. According to Proposition~\ref{prop:Galh}, this element $\sigma$ is uniquely given as the prolongation of $x\mapsto \frac{1-8x}{8+8x}$ with $\sigma(h^2)=u\cdot h^2\cdot (x+1)^2$, where the prefactor $u$ has to be $\frac89$. We have 
        \[\textstyle \sigma(f)= \sigma(h^2) (\sigma(x)+1) = \frac 8 9 h^2 (x+1)^2 (\sigma(x)+1) = h^2\cdot (x+1) = f.\]
        So indeed, $f$ is fixed by $\sigma$, and the extension has degree $e$.
        \item If $p\equiv 11 \pmod{24}$ one proceeds as for $7$, as $u=\frac 8 9$ in the definition of the involution $\sigma$.     
        \item If $p\equiv 13 \pmod{24}$, the group $\Gal(\F_p(x, h^2)/\F_p(t))$ is cyclic, and the unique element of order $2$ in the Galois group is an element of $S$, sending $h^2$ to $-h^2$. Thus $f$ is not fixed and $\F_p(t, f)\simeq \F_p(x, h^2)$, having degree $p{-}1$ over $\F_p(t)$.
        \item The case $p\equiv 17 \pmod{24}$ works as $p\equiv 13 \pmod{24}$, with the same conclusion.
        \item If $p\equiv 19 \pmod{24}$, we proceed as in the case $p\equiv 7 \pmod{24}$, with the exception that $u=-\frac 8 9$. Thus $\sigma(f)=-f$, and $f$ is not in the unique subfield of index $2$ of $\F_p(x, h^2)$. Thus $\F_p(t, f)$ has degree $p{-}1$.
        \item If $p\equiv 23 \pmod{24}$ we proceed as for $p\equiv 19 \pmod{24}$, as $u=-\frac 8 9$ again.
    \end{itemize}
This concludes the proof.
\end{proof}

\begin{proof}[Proof of Theorem~\ref{thm:Ap}.]
We recall that we have the relation $f = (1+x) h^2$ in $\F_p(t,f)$.
Therefore $f^e = (1+x)^e H \in \F_p(x)$, from which we deduce that $A_p = (1+x)^{p-1} H^2$. In particular, $A_p$ is a square in $\F_p(x)$.
Since moreover $\F_p(x)=\F_p(t, \sqrt{\Delta})$ with $\Delta = t^2-34t+1$, we can write $A_p=(u+v\sqrt{\Delta})^2$ for $u, v\in \F_p(t)$.
As $A_p\in \F_p(t),$ we must necessarily have that $2uv=0$.
Hence either $u = 0$, in which case $A_p = (t^2-34t+1) v^2$, or $v = 0$, in which case $A_p = u^2$.
Given that $A_p$ is a square if $\F_p(t)$ if and only if the extension $\F_p(t,f)/\F_p(t)$ has degree $e$, we conclude using Theorem~\ref{thm:apery}.
\end{proof}

\section{Adaptation to Apéry-like Sequences}

The \emph{generalized} Apéry numbers $a_n(r,s)=\sum_{k=0}^n \binom{n}{k}^r \binom{n+k}{n}^s$ are $p$-Lucas for all primes $p$ \cite{DS06}. So we have $f_{r,s}=A_{r,s} f_{r,s}^p(x)$ for the generating series $f_{r,s}$ of $a_n(r,s)$ and its truncation $A_{r,s}$ at order $p$. For $(r,s)\in \{(0,0), (1, 0), (0,1), (1,1), (2,0)\}$ the generating function $f_{r,s}$ is algebraic (for $(r,s)=(1,1)$ it is given by $(2x(\sqrt{1-4x}-1))^{-1}$,  for $(r,s)=(1,1)$ it is given by $\sqrt{1 - 6t + t^2}^{-1}$ and for $(r,s)=(1,1)$ it is given by $\sqrt{1-4x}^{-1}$). Only for the case $(r,s)\in\{(2,2), (4,0)\}$ we (computationally) observed that $A_{r,s}$ is a square for a class of prime numbers depending on some congruence conditions: the case $(r,s)=(2,2)$ are the regular Apéry numbers, and $(4,0)$ is studied in Section~\ref{ssec:zoo}, see level $10$ in Table~\ref{tab:2}. So generalizing the investigation in this direction does not show many new interesting patterns.

As explained in the introduction, the Domb numbers and the AZ numbers are closely related to Apéry numbers. The proofs of very similar patterns, as stated in Theorems~\ref{thm:domb} and~\ref{thm:az}, proceed analogously to the Apéry case.
We report in Sections~\ref{ssec:domb} and~\ref{ssec:AZ} on the changes that need to be made. 

In the last section, Section~\ref{ssec:zoo}, more sequences of similar flavor are analyzed.

\subsection{The Domb Numbers}
\label{ssec:domb}

We set $t_\delta\coloneqq \frac{x(x+1)}{1-8x}$. Then the relation to $h(x)$ is given by $\frac 1 {1-8x} f_\delta(t_\delta)=h(x)^2$. In Lemma~\ref{lem:extensionxt} we replace the involution $\sigma$ by $\sigma_\delta:x\mapsto \frac{1+x}{8x-1}$, and the generator of the field extension $\F_p(x)/\F_p(t_\delta)$ is $(64t_\delta^2-20t_\delta+1)^{1/2}$.

As in Lemma~\ref{lem:sigmaH} we obtain

\[H=\begin{cases}
        \sigma_\delta(H)\cdot (8x-1)^{p-1} & \text{if $p\equiv 1 \pmod 6$}\\ - \sigma_\delta(H)\cdot (8x-1)^{p-1} & \text{if $p\equiv 5 \pmod 6$.}
    \end{cases}\]

In Lemma~\ref{lem:extsigma}, the automorphism $\sigma_\delta$ can be extended as $\sigma(h^2)=u \cdot h^2 \cdot (8x-1)^2$ for $u$ being either a square or a non-square, depending on the congruence class of $p\bmod 6$. Such a prolongation is an involution on $\F_p(x, h^2)$ if and only if one can take $u=\pm \frac 1 9$. One of these values always is a square, while the other one is a square if and only if $p\equiv 1 \pmod 4$. Thus,  
    \[\Gal(\F_p(x, h^2)/\F_p(t_\delta)) = 
        \begin{cases}
            \Z/(p{-}1)\Z\simeq \Z/e\Z\times \Z/2\Z & \text{if } p\equiv 7, 11 \pmod{12}\\
            \Z/(p{-}1)\Z & \text{if } p\equiv 5 \pmod{12}\\
            \Z/e\Z\times \Z/2\Z& \text{if } p\equiv 1 \pmod{12},
        \end{cases}
    \]
again corresponding to the case of one, zero, respectively two extensions of $\sigma_\delta$ to an involution on $\F_p(x, h^2)$. 

Finally, $f_\delta \in \F_p(x, h^2)$ because of the relations stated above and:
\begin{itemize}
    \item If $p\equiv 1\pmod{12}$ then the Galois group of $\F_p(x, h^2)/\F_p(t_\delta)$ is not cyclic and $\F_p(t_\delta, f_\delta)$ is a subfield of $\F_p(x, h^2)$ of order $e$.
    \item If $p\equiv 5\pmod{12}$ then the Galois group $\F_p(x, h^2)/\F_p(t_\delta)$ is cyclic. The question then becomes to determine whether $f_\delta$ is fixed by $\sigma_\delta$. However, the unique element of order $2$ in the group in this case belongs to the group $S$ sending $h^2$ to $-h^2$ and thus $f_\delta$ is not fixed.
    \item If $p \equiv 7 \pmod{12}$ then the Galois group of $\F_p(x, h^2)/\F_p(t_\delta)$ is cyclic. There is one involution extending $\sigma$, corresponding to $u=\frac 1 9$. We compute
    \[\textstyle \sigma_\delta(f_\delta)= \sigma_\delta(h^2) (1-8\sigma_\delta(x)) = \frac 1 9 h^2 (1-8x)^2 (1-8\sigma_\delta(x)) = h^2\cdot (1-8x) = f_\delta,\]
    so $f_\delta$ is fixed and thus contained in a proper subfield of $\F_p(x, h^2)$ of order $e$.
    \item If $p \equiv 11 \pmod{12}$ the prolongation of $\sigma_\delta$ corresponds to $u=-\frac 1 9$, and this time $f_\delta$ changes sign under $\sigma_\delta$.
\end{itemize}

\subsection{The Almkvist--Zudilin Numbers}
\label{ssec:AZ}

For the AZ-numbers we set $t_\xi\coloneqq \frac x {(1+x)(1-8x)}$ and find the involution $\sigma_\xi\coloneqq x\mapsto -\frac 1 {8x}$ of $\F_p(x)$ fixing $\F_p(t_\xi)$. This time we obtain the simple relation $H=\sigma(H)\cdot x^{p-1}$ for all prime numbers $p$. The involution $\sigma_\xi$ can be extended to $\F_p(x, h^2)$ by $\sigma(h)^2=u h^2 x^2$ with $u\in S$; it is an involution if and only if we can take $u=\pm 8$. Now $8$ is a square if and only if $p\equiv \pm 1 \pmod 8$ and $-1$ is a square if and only if $p\equiv 1, 3 \pmod 5$. Thus 
\[\Gal(\F_p(x, h^2)/\F_p(t_\xi)) = 
        \begin{cases}
            \Z/(p{-}1)\Z\simeq \Z/e\Z\times \Z/2\Z & \text{if } p\equiv 3,7 \pmod 8\\
            \Z/(p{-}1)\Z & \text{if } p\equiv 5 \pmod 8\\
            \Z/e\Z\times \Z/2\Z& \text{if } p\equiv 1 \pmod 8.
        \end{cases}
    \]
Finally, $f_\xi \in \F_p(x, h^2)$ and investigations depending on the congruence class of $p\pmod 8$ yield the following:
\begin{itemize}
    \item If $p\equiv 1 \pmod 8$ the Galois group $\Gal(\F_p(x, h^2)/\F_p(t_\xi))$ is not cyclic, and thus $f_\xi$ is in a proper subfield of degree $e$.
    \item If $p\equiv 3 \pmod 8$ the unique prolongation of $\sigma_\xi$ as an involution corresponds to $u=-8$ and because of 
    \[\textstyle \sigma_\xi (f_\xi)= \sigma(h^2) (1+\sigma_\xi(x))(1-8\sigma_\xi(x)) =  h^2 (1-8x)(1+x)= f_\xi,\]
    we conclude that $f_\xi$ lies in a proper subfield of $\F_p(x, h^2)$ of degree $e$.
    \item If $p\equiv 5 \pmod 8$ the unique element of order $2$ in $\Gal(\F_p(x, h^2)/\F_p(t_\xi)$ belongs to the subgroup of squares of $\F_p^\times$ and sends $h$ so $-h$ and thus does not fix $f_\xi$. So  $\F_p(t_\xi, f_\xi)=\F_p(x, h)$.
    \item If $p\equiv 7 \pmod 8$ the prolongation of $\sigma_\xi$ as an involution corresponds to $u=+8$, and this time $f_\xi$ changes sign under it and thus is not fixed by it. Again, $\F_p(t_\xi, f_\xi)=\F_p(x, h)$.
\end{itemize}

\subsection{A Zoo of Further Examples} \label{ssec:zoo}

First, we consider the three other sequences of  Zagier's sporadic examples of integral sequences satisfying a three-term recurrence relation~\cite{Zag09}.
It is known that these sequences are $p$-Lucas for all primes $p$~\cite{Malik-Straub}; therefore, their generating series $f(t)$ satisfies the relation $f(t) \equiv A_p(t) f(t)^p \pmod p$, where $A_p(t) \in \F_p[t]$ is the truncation at $t^p$ of $f(t) \bmod p$. As before, we are interested in the factorization of $A_p(t)$ of the form $A_p(t) = P(t) B_p(t)^2$.
Table~\ref{tab:1} shows how $P(t)$ varies with respect to $p$.

\newcommand{\sequence}[2]{
  \multirow[t]{1}*{ 
  \renewcommand\arraystretch{1}
  \begin{tabular}{@{}c@{}} ~ \\[2ex]
    \href{https://oeis.org/#1}{OEIS #1} \vspace{0.5ex} \\
    {\small $#2$}
  \end{tabular}}
}

\renewcommand\arraystretch{1.8}

\begin{figure}
\centering
\begin{tabular}{|c|c|c|c|}
     \hline
     Sequence & Conditions & Coefficient $P(t)$ \\
     \hline \hline

     \sequence
       {A229111}
       {2\cdot (-1)^n\cdot \sum_{k=0}^n \binom n k^3 \big(\binom{4n-5k-1}{3n}{+}\binom{4n-5k}{3n}\big)}
     & $\legendre{-1}{p}= 1$  &  $1$ \\
     & $\legendre{-1}{p}=-1$  &  $1{-}22t{+}125t^2$\\

     \hline

     \sequence
       {A290575}
       {\sum_{k=0}^n \binom n k^2 \binom{2k}{n}^2}
     & $\legendre{-2}{p}=1$ &  $1$ \\
     & $\legendre{-2}{p}=-1$ & $1{-}24t{+}16t^2$\\

     \hline

     \sequence
       {A290576}
       {\sum_{k=0}^n \sum_{\ell=0}^n \binom n k^2 \binom n \ell \binom k \ell \binom{k+\ell}{n}}
     & $\legendre{-1}{p}=1$ &  $1$ \\
     & $\legendre{-1}{p}=-1$ & $1{-}18t{+}27t^2$\\

     \hline

\end{tabular}
\caption{Zagier's sporaric examples~\cite{Zag09,AVZ11}} \label{tab:1}
\end{figure}

We pursue our investigations with other sequences associated to modular functions (of a given level), see Table~\ref{tab:2}.
Thanks to~\cite{ABD19,BTY25}, we know that all the series appearing in this table are $p$-Lucas except for levels $17$, $20$ and $23$.
However, in the former cases, they continue to satisfy a relation of the form $f(t) \equiv A_p(t) f(t)^p \pmod p$, where now $A_p(t) \in \F_p(t)$ is now a rational function; this follows from the main theorem of~\cite{Var23}.
As a conclusion, in all cases, $A_p(t)$ is well defined and it makes sense to study its factorization as $A_p(t) = P(t) B_p(t)^2$ as before.

In all the examples of Tables~\ref{tab:1} and~\ref{tab:2}, we computationally observed that the patterns for the polynomials $P(t)$ depend on explicit quadratic residues conditions for the prime~$p$.
 These patterns are very much in line with those recorded in Theorems \ref{thm:apery}, \ref{thm:Ap}, \ref{thm:domb} and \ref{thm:az}, and one can apply, case by case, a similar strategy to establishing the claims in the tables rigorously.
The principal feature of the underlying series $f(t)$ is the presence of suitable rational parameterizations $t=t(x)$ such that $f(t(x))=\rho(x)h(x)^2$ for a function $h(x)$ solving a second order linear differential equation and in turn related to an arithmetic $_2F_1$ hypergeometric function $g(y)$ via another parameterization $y=y(x)$; namely, $h(x)=\lambda(x)g(y(x))$.
In Ap\'ery's case (but also for the Domb and Almkvist--Zudilin sequences), all the intermediate functions are rational in~$x$:
\[
t(x)=\frac{x(1-8x)}{1+x}, \quad \rho(x)=1+x, \quad y(x)=\frac{27x^2}{(1-2x)^3}, \quad \lambda(x)=\frac1{1-2x}.
\]
However, in general $\rho(x)$, $y(x)$ and $\lambda(x)$ are algebraic, which can make the structure of the corresponding Galois group more involved.
The existence of such $x$-parameterizations is a consequence of modular parameterizations of all such Ap\'ery-like sequences; the details of the latter can be found in the corresponding references where our examples originate from. We leave the details to an interested reader as exercises.

Going further, the observations on the splitting pattern of $A_p(t)$ as $P(t)B_p(t)^2$ suggest that the conditions of $p$ depends only on the quadratic residues of the divisors of the squarefree part of the corresponding level.
We also leave it to the interested reader to formulate and prove a precise statement about the pattern that unfolds here. We conclude by remarking that these observations provide more computational evidence for the conjectures on uniformity properties of Galois groups of reductions of D-finite series, that were formulated in~\cite{CFV25}. 
 
\renewcommand{\sequence}[3]{
  \multirow[t]{1}*{ 
  \renewcommand\arraystretch{1}
  \begin{tabular}{@{}c@{}} ~ \\[1ex]
    \href{https://oeis.org/#1}{OEIS #1} \vspace{-0.5ex} \\
    {\scriptsize #2} \vspace{-0.5ex}\\
    {\scriptsize #3}
  \end{tabular}}
}
\newcommand{\sequencewithoutOEIS}[3]{
  \multirow[t]{1}*{ 
  \renewcommand\arraystretch{1}
  \begin{tabular}{@{}c@{}} ~ \\[1ex]
    {\scriptsize #2} \vspace{-0.5ex}\\
    {\scriptsize #3}
  \end{tabular}}
}
\newcommand{\sequenceformula}[3]{
  \multirow[t]{1}*{ 
  \renewcommand\arraystretch{1}
  \begin{tabular}{@{}c@{}} ~ \\[2ex]
    \href{https://oeis.org/#1}{OEIS #1} \vspace{-0.5ex} \\
    {\scriptsize #2} \\
    {\small $#3$}
  \end{tabular}}
}

\begin{figure}
\centering
\begin{tabular}{|@{}c@{}|c|c|c|}
     \hline
     Sequence & Level & Conditions & Coefficient $P(t)$ \\
     \hline \hline

     \sequenceformula
       {A274786}
       {\cite{Coo12}, \cite{BBMW15}, \cite{Coo17}, \cite{HSY23}}
       {\binom{2n} n \sum_{k=0}^n \binom n k^2 \binom {n+k} k}
      & 5 & $\legendre{-5}{p}=1$ & $1$ \\
      &   & $\legendre{-5}{p}=-1$ & $1{-}44t{-}16t^2$ \\

     \hline

     \sequenceformula
       {A181418}
       {\cite{Coo12}, \cite{Coo17}, \cite{HSY23}}
       {\binom{2n} n \sum_{k=0}^n \binom n k^3}
      & 6 & $\legendre{-3}{p} = \legendre{-6}{p} = 1$ & $1$ \\
      &   & $\legendre{-3}{p} = -1, \legendre{-6}{p} = 1$ & $1{+}4t$ \\
      &   & $\legendre{-3}{p} = 1, \legendre{-6}{p} = -1$ & $1{-}32t$ \\
      &   & $\legendre{-3}{p} = \legendre{-6}{p} = -1$ & $(1{+}4t)(1{-}32t)$ \\

     \hline

     \sequenceformula
       {A183204}
       {\cite{Coo12}, \cite{Coo17}, \cite{HSY23}}
       {\sum_{k=0}^n \binom n k^2 \binom {2k} n \binom {k+n} n}
      & 7 & $\legendre{-7}{p}=1$ & $1$ \\
      &   & $\legendre{-7}{p}=-1$ & $(1{+}t) (1{-}27t)$ \\

     \hline

     \sequenceformula
       {A005260}
       {\cite{Coo12}, \cite{Coo17}, \cite{HSY23}}
       {\sum_{k=0}^n \binom n k^4}
      &10 & $\legendre{-5}{p} = \legendre{-10}{p} = 1$ & $1$ \\
      &   & $\legendre{-5}{p} = -1, \legendre{-10}{p} = 1$ & $1{+}4t$ \\
      &   & $\legendre{-5}{p} = 1, \legendre{-10}{p} = -1$ & $1{-}16t$ \\
      &   & $\legendre{-5}{p} = \legendre{-10}{p} = -1$ & $(1{+}4t)(1{-}16t)$ \\

     \hline

     \sequence
       {A284756}
       {\cite[$c_{11}$ Thm.~4.7]{CGY15}}
       {\cite{Coo17}, \cite{HSY23}}
      &11 & $\legendre{-11}{p}=1$ & $1$ \\
      &   & $\legendre{-11}{p}=-1$ & $1{-}20t{+}56t^2{-}44t^3$ \\

     \hline

     \sequencewithoutOEIS
       {}
       {\cite[Cor~3.6]{HSY20}}
       {}
      &17 & $\legendre{-17}{p}=1$ & $1$ \\
      &   & $\legendre{-17}{p}=-1$ & $1{-}16t{-}66t^2{-}48t^3{-}127t^4$ \\

     \hline

     \sequence
       {A219692}
       {\cite[$s_{18}$ on p.171]{Coo12}}
       {}
      &18 & $\legendre{-1}{p} = \legendre{-2}{p} = 1$ & $1$ \\
      &   & $\legendre{-1}{p} = -1, \legendre{-2}{p} = 1$ & $1{-}12t$ \\
      &   & $\legendre{-1}{p} = 1, \legendre{-2}{p} = -1$ & $1{-}16t$ \\
      &   & $\legendre{-1}{p} = \legendre{-2}{p} = -1$ & $(1{-}12t)(1{-}16t)$ \\

     \hline

     \sequencewithoutOEIS
       {}
       {\cite[Cor~3.7]{HSY18}}
       {}
      &20 & $\legendre{-1}{p} = \legendre{-5}{p} = 1$ & $1$ \\
      &   & $\legendre{-1}{p} -1, = \legendre{-5}{p} = 1$ & $1{-}4t$ \\
      &   & $\legendre{-1}{p} = 1, \legendre{-5}{p} = -1$ & $1{-}12t{+}16t^2$ \\
      &   & $\legendre{-1}{p} = \legendre{-5}{p} = -1$ & $(1 -4t) (1{-}12t{+}16t^2)$ \\

     \hline

     \sequencewithoutOEIS
       {}
       {\cite[$c_{23}$ Thm.~4.7]{CGY15}, \cite{Coo17}}
       {}
      &23 & $\legendre{-23}{p}=1$ & $1$ \\
      &   & $\legendre{-23}{p}=-1$ & \small $(1{-}3t{+}2t^2{+}t^3)(1{-}11t{+}22t^2{-}19t^3)$ \\

     \hline
\end{tabular}
\caption{Examples of sequences connected to modular forms} \label{tab:2}
\end{figure}
\newpage
\printbibliography

\textsc{CNRS; IMB, Université de Bordeaux, 351 Cours de la Libération, 33405 Talence,
France}\\
\textit{Email: }\href{mailto:xavier@caruso.ovh}{\texttt{xavier@caruso.ovh}}\\

\textsc{Faculty of Mathematics, University of Vienna, Oskar-Morgenstern-Platz 1, 1090 Vienna, Austria}\\
\textit{Email: }\href{mailto:florian.fuernsinn@univie.ac.at}{\texttt{florian.fuernsinn@univie.ac.at}}\\

\textsc{Institut de Mathématiques de Toulouse, Université Paul Sabatier, 118 route de Narbonne, 31062 Toulouse Cedex 9, France
}\\
\textit{Email: }\href{mailto:daniel.vargas-montoya@math.univ-toulouse.fr}{\texttt{daniel.vargas-montoya@math.univ-toulouse.fr}}\\

\textsc{IMAPP, Radboud University Nijmegen, PO Box 9010, 6500 GL Nijmegen, The Netherlands}\\
\textit{Email: }\href{mailto:w.zudilin@math.ru.nl}{\texttt{w.zudilin@math.ru.nl}}\\

\end{document}